\documentclass[12pt]{article}
\usepackage{latexsym}
\usepackage{amsmath}
\usepackage{euscript}
\usepackage{amssymb}
\usepackage{amsthm}
\usepackage{amsopn}
\usepackage{graphics,color,url}
\newtheorem{prop}{Proposition}[section]
\newtheorem{thm}[prop]{Theorem}

\newtheorem{cor}[prop]{Corollary}

\newcommand{\stirr}[2]{\genfrac{\{}{\}}{0pt}{}{#1}{#2}}

\title{Avoiding Colored Partitions\\ of Lengths Two and Three}
\author{Adam M. Goyt\\Department of Mathematics\\Minnesota State University Moorhead\\Moorhead, MN 56563\\ goytadam@mnstate.edu\and Lara K. Pudwell\\Department of Mathematics and Computer Science\\Valparaiso University\\ Valparaiso, IN 46383\\Lara.Pudwell@valpo.edu}

\begin{document}

\bibliographystyle{elsart-num-sort}

\maketitle

\begin{abstract}
Pattern avoidance in the symmetric group $S_n$ has provided a number
of useful connections between seemingly unrelated problems from
stack-sorting to Schubert varieties.  Recent work~\cite{EggeColored,MansourColoured,MansourColoured2,MansourWest} has generalized these results to $S_n\wr C_c$, the objects of which can be viewed as ``colored permutations''.

Another body of research that has grown from the study of pattern
avoidance in permutations is pattern avoidance in $\Pi_n$, the set
of set partitions of $[n]$.  Pattern avoidance in set partitions is a generalization of the well-studied notion of noncrossing partitions~\cite{KrewerasNonCross}.  

Motivated by recent results in pattern avoidance in $S_n \wr C_c$ we provide a catalog of initial results for pattern avoidance
in colored partitions, $\Pi_n \wr C_c$.  We note that
colored set partitions are not a completely new concept.
\emph{Signed} (2-colored) set partitions appear in the work of
Bj\"{o}rner and Wachs involving the homology of partition lattices
\cite{BjornerWachsColored}. However, we seek to study these objects in a new
enumerative context.

\end{abstract}

\section{Introduction}

Pattern avoidance in permutations has been a very popular area of study since it was introduced by Knuth~\cite{Knuthvol3} and expanded by Simion and Schmidt~\cite{SimionSchmidt}.  About ten years later pattern avoidance in set partitions was studied by Klazar~\cite{KlazarPartI} and Sagan~\cite{SaganPatternAvoidance}.  Recently pattern avoidance in colored permutations has been studied by Mansour~\cite{MansourColoured,MansourColoured2} and Egge~\cite{EggeColored}.  In this paper, we consider the idea of pattern avoidance in colored set partitions and focus on enumerating sets of colored set partitions which avoid set partitions of length two or three.  

A {\it partition} of the set $[n]=\{1,2,\dots,n\}$ is a family of pairwise disjoint subsets $B_1$, $B_2,\dots B_k$ of $[n]$, whose union, $\uplus_{i=1}^k B_i=[n]$.  We refer to the subsets as blocks, and write our set partitions with the blocks in canonical order $B_1/B_2/\dots/B_k$ with $\min B_1<\min B_2<\dots<\min B_k$.  For example, $137/25/4/6$ is the partition of $[7]$ with blocks $\{1,3,7\}$, $\{2,5\}$, $\{4\}$, and $\{6\}$.  We denote the set of all partitions of $[n]$ by $\Pi_n$. We write $P \vdash [n]$ if $P$ is a partition of $[n]$. 

Associated to each set partition $P \vdash [n]$ is a canonical word in $[k]^n$, of the form $a_1a_2\dots a_k$, where $a_i=j$ if and only if $i\in B_j$.  For example, $137/25/4/6$ is associated with the word $1213241$.  

Given any word $w\in[k]^n$, $w$ can be {\it canonized} by simply replacing each occurrence of the first occurring letter in $w$ by 1, each occurrence of the second occurring letter in $w$ by 2, etc.  For example, the word $4223472774$ is canonized to the word $1223142441$.  

Let $P \vdash [k]$ and $Q \vdash[n]$, we say that $Q$ contains a copy of $P$ if there exists a subsequence of $Q$ of length $k$ whose canonization is $P$.  Otherwise, we say that $Q$ avoids $P$.  For example, let $P_1=112$ and $Q=1213241$.  Then $113$ and $221$ are both copies of $P_1$ in $Q$.  On the other hand, $Q$ avoids the partition $P_2=1222$.  

{\it Colored partitions} of $[n]$ are partitions of $[n]$ where each element in the partition is given one of $c$ colors.  We will denote by $\Pi_n\wr C_c$ the set of all colored set partitions borrowing notation from the notion of the wreath product of groups.  Consider the partition $1213241\in\Pi_7$.  If we use colors from the set $\{1,2,3\}$ then $1^22^11^13^12^34^31^2$ is a colored partition from $\Pi_7\wr C_3$.  

There are a few ways that we can study pattern avoidance in colored set partitions.  Let $P\in\Pi_k\wr C_{c_1}$ and $Q\in\Pi_n\wr C_{c_2}$.  We say that $Q$ contains a copy of $P$ in the $eq$ sense if the uncolored version of $Q$ contains a copy of the uncolored version of $P$, and the colors on this copy match the colors on $P$.  Otherwise we say that $Q$ {\it eq-avoids} $P$.  

We say that $Q$ contains a copy of $P$ in the $lt$ sense if the uncolored version of $Q$ contains a copy of the uncolored version of $P$, and the colors of this copy are element-wise less than or equal to the colors of $P$.  Otherwise we say that $Q$ {\it lt-avoids} $P$.  

We say that $Q$ contains a copy of $P$ in the $pattern$ sense if the uncolored version of $Q$ contains a copy of the uncolored version of $P$, and the colors of this copy form the same set partition pattern as the colors of $P$.  Otherwise we say that $Q$ {\it pattern-avoids} $P$.  

Consider $Q=1^22^11^13^12^34^31^2$ and the patterns $P_1=1^12^43^4$ and $P_2=1^12^11^3$.  The subsequence $1^22^34^3$ is a copy of $P_1$ in the $pattern$ sense.  The subsequence $2^11^12^3$ is a copy of $P_2$ in the $eq$ sense.  The subsequence $1^12^34^3$ is a copy of $P_1$ in the $pattern$ sense and the $lt$ sense, but not the $eq$ sense.  In fact $Q$ {\it eq-avoids} $P_1$.  

Let $c_1$ and $c_2$ be positive integers and let $R\subset \cup_{k=1}^\infty \Pi_k\wr C_{c_1}$, then $\Pi^{eq}_n\wr C_{c_2}(R)$ is the set of all colored partitions in $\Pi_n\wr C_{c_2}$ which $eq$-avoid every partition in $R$.  Noting that any pattern in the $lt$ sense or in the $pattern$ sense can be written as a \emph{set} of patterns in the $eq$ sense, in this paper we will focus on enumerating the sets $\Pi^{eq}_n\wr C_{c_2}(R)$, where $R$ is a subset of either $\Pi_2\wr C_{c_1}$ or $\Pi_3\wr C_{c_1}$.  In Section 2 we enumerate colored partitions that $eq$-avoid a single colored partition of length two.  In Section 3 we develop bijective proofs involving partitions that avoid multiple patterns from $\Pi_2\wr C_2$.  We then determine Wilf classes for colored patterns of length three in the fourth section.  In the fifth section we enumerate some of these classes.  Finally, we discuss ideas for future research in this area.

\section{Colored Patterns of Length 2}

In this section we consider partitions of $\Pi_n \wr C_c$ that avoid a single pattern of length 2.  There are 4 such types of patterns:  $1^a1^a$, $1^a1^b$, $1^a2^a$, and $1^a2^b$, where $a$ and $b$ denote distinct colors $ 1 \leq a,b \leq c$.  A summary of the results of this section appears in Table \ref{Tlength2}.  Even though there are 4 types of patterns of length 2, they only produce 3 distinct enumeration sequences.

Throughout this paper we will denote the $n^{th}$ Bell number by $B(n)$ and we will denote the Stirling numbers of the second kind by $\stirr{n}{k}$.  We also note that if $c$ is less than the number of colors in pattern $P$ or if $n$ is less than the length of pattern $P$, then $\left|\Pi_n^{eq} \wr C_c(P)\right|=\left|\Pi_n \wr C_c\right|=c^nB(n)$.  All theorems throughout this paper hold for $n$ and $c$ sufficiently large, that is, where $n$ is at least the length of the forbidden pattern $P$ and $c$ is at least the number of distinct colors used in forbidden pattern $P$.

\begin{thm} 
$$\left|\Pi_n^{eq} \wr C_c(1^a1^a) \right| =  \sum_{i+2j \leq n} \binom{n}{i,j} \sum_{p=j}^{n-i-j} \binom{n-i-j}{p} \stirr{p}{j} j! B(n-i-j-p) (c-1)^{n-i-j}$$
where $a$ is any color $1 \leq a \leq c$.
\end{thm}
\begin{proof}
A colored partition avoids this pattern exactly when there are no two elements in the same block that are both colored with color $a$.

Let $i+j$ be the number of elements of the partition that will be colored with color $a$.  Notice that each of these $i+j$ elements will be in separate blocks of the partition.  Choose $i$ of them to be singleton blocks, and $j$ of them to be in blocks shared with numbers of other colors.  There are $\binom{n}{i,j}$ ways to choose the integers that will play these two separate roles.

Now, let $p$ be the number of elements of the partition that will be colored with a color other than $a$ but will share a block with one of the $j$ $a$-colored elements already chosen.  Clearly, $p$ must be at least $j$, and it can be at most $n-i-j$.  There are $\binom{n-i-j}{p}$ ways to choose which elements will share blocks with the $j$ $a$-colored elements, $\stirr{p}{j}$ ways to partition these $p$ elements into $j$ blocks, $j!$ ways to decide which of the $j$ blocks will be matched with which of the $j$ $a$-colored elements, and $(c-1)^p$ ways to color these $p$ elements.

After we have colored $i+j$ elements with color $a$ and partitioned and colored $p$ additional elements to share blocks with these elements, there remain $n-i-j-p$ elements to partition in any way we choose and color with any color besides $a$.  This can be done in $B(n-i-j-p) (c-1)^{n-i-j-p}$ ways.

Combining all of these steps yields the above formula.
\end{proof}

\begin{thm} $\left|\Pi_n^{eq} \wr C_c (1^a1^a)\right|=\left|\Pi_n^{eq} \wr C_c (1^a1^b)\right|$ for all $n \geq 0$ and $c \geq 2$. \end{thm}

\begin{proof}
A partition avoids the pattern $1^a1^a$ exactly when there are no two $a$-colored elements in the same block, whereas a partition avoids the pattern $1^a1^b$ exactly when there is no $a$-colored element followed by a $b$-colored element in the same block.

Clearly all partitions with no $a$-colored elements avoid each of these patterns so we will focus on partitions where there are $a$-colored elements.  

Consider a block of size $k$ in a partition that avoids $1^a1^a$ with an $a$-colored element.  There are $k$ choices for which element may be colored $a$, then there are $(c-1)^{k-1}$ ways to color the remaining elements with any color except for $a$, yielding $k(c-1)^{k-1}$ ways to color the block.

Consider a block of size $k$ in a partition that avoids $1^a1^b$ with an $a$-colored element.  There are $k$ choices for which element may be the smallest $a$-colored element, then there are $(c-1)^{k-1}$ ways to color the remaining elements as elements smaller than the chosen one may not have color $a$ and elements larger than the chosen one may not have color $b$.

Since there are the same number of ways to color a block to avoid $1^a1^a$ as there are to color a block to avoid $1^a1^b$, the two sets in question are equinumerous.
\end{proof}

\begin{thm} \label{theorem2ndbells}
$$\left|\Pi_n^{eq} \wr C_c (1^a2^a) \right| = B(n) (c-1)^n + \sum_{i=1}^n \sum_{j=0}^{n-i} \binom{n}{i} \binom{n-i}{j} B(n-i-j) (c-1)^{n-i}$$
where $a$ is any color $1 \leq a \leq c$.
\end{thm}
\begin{proof}
Notice that a partition avoids the pattern $1^a2^a$ exactly when all elements colored with color $a$ lie in the same block of the partition.

If there are no $a$-colored elements, then there are $B(n)$ ways to partition $\{1,\dots,n\}$ and $(c-1)^n$ ways to color these elements without using the color $a$.

If there are $a$-colored elements, then we choose $i$ elements to have color $a$ in $\binom{n}{i}$ ways, then choose $j$ other elements to share a block with these elements in $\binom{n-i}{j}$ ways.  Now, there are $B(n-i-j)$ ways to partition the remaining elements into other blocks, and $(c-1)^{n-i}$ ways to color the non $a$-colored elements.

Summing over all possible values for $i$ and $j$ yields the above formula.
\end{proof}

We note that when $c=2$, this formula reduces to $B(n)+\sum_{i=1}^n \sum_{j=0}^{n-i}\binom{n}{i,j} B(n-i-j)=B(n+2)-2B(n+1)+B(n)$, the second differences of the Bell numbers (A011965).

Unfortunately $\left| \Pi_n^{eq} \wr C_c (1^a2^b)\right|$ is not quite a simple as one might hope from earlier arguments in this section.  The exact formula given in Theorem \ref{theorem1a2b} gives a taste of the arguments to come in Section \ref{seclen3}.

\begin{thm}\label{theorem1a2b}
$\displaystyle{\left| \Pi_n^{eq} \wr C_c (1^a2^b)\right| = }$
$$B(n)(c-1)^n + \sum_{i=0}^n \sum_{k=0}^{n-i} \sum_{\ell=0}^{i-1} \binom{n-i}{k} \binom{i-1}{\ell}B(n-k-\ell-1)c^k(c-1)^{i-1}(c-2)^{n-i-k} +$$ $$\sum_{i=0}^{n-1}\sum_{j=i+1}^n \sum_{p+q \leq i-1} \sum_{r+s \leq j-i-1} \sum_{t+u \leq n-j} \binom{i-1}{p,q}\binom{j-i-1}{r,s}\binom{n-j}{t,u}\cdot$$ $$B(n-p-q-r-s-t-u-2)c^r(c-1)^{n-j+i-1}(c-2)^{j-i-r-1}$$
for $1 \leq a \neq b \leq c.$
\end{thm}

\begin{proof}
The counting argument falls into 3 cases: (i) There are no elements with color $a$. (ii) There is at least one element with color $a$ and all elements with color $a$ are in one block.  (iii) There is at least one element with color $a$, and elements with color $a$ can be found in at least two different blocks.

In case (i), it is impossible to form a copy of $1^a2^b$ with no $a$-colored elements, so we count all possible $B(n)(c-1)^n$ partitions.

In case (ii), let $i$ be the smallest $a$-colored element.  Choose $k$ of the $n-i$ elements larger than $i$ and $\ell$ of the $i-1$ elements smaller than $i$ to be in the same block as $i$.  Partition the remaining $n-1-k-\ell$ elements in $B(n-1-k-\ell)$ ways.

Now, note that elements larger than $i$ in the same block as $i$ may have any color.  Elements smaller than $i$ may have any color other than $a$.    Elements larger than $i$ in a different block may have any color except $a$ or $b$.  Summing over all possible values of $i$, $k$, and $\ell$ yields the second term above.

Finally, in case (iii), let $i$ be the smallest $a$-colored element, and let $j$ be the smallest $a$-colored element in a different block than $i$.  Of the elements less than $i$ we choose $p$ to go in the block with $i$ and $q$ to go in the block with $j$.  Of the elements between $i$ and $j$ we choose $r$ to go in the block with $i$ and $s$ to go in the block with $j$.  Of the elements larger than $j$ we choose $t$ to go in the block with $i$ and $u$ to go in the block with $j$.

Now, note that the $r$ elements between $i$ and $j$ in the same block as $i$ may have any color.  Elements less than $i$ may have any color but $a$ and elements greater than $j$ may have any color but $b$.  Elements between $i$ and $j$ that are not in the same block as $i$ may not have color $a$ or color $b$.  Summing over all reasonable values of our variables yields the above result.
\end{proof}

In the case of $c=2$, this formula ``simplifies'' in that we require one fewer variable in case (ii), and two fewer variables in case (iii).

\begin{table}[hbt]
\begin{center}
\begin{tabular}{|c|c|c|}
\hline
Pattern $P$& First 8 terms of $\left|\Pi_n \wr C_2(P)\right|$&OEIS number\\
\hline
$1^a1^a$&2, 7, 30, 152, 878, 5653, 39952, 306419&new\\
\hline
$1^a1^b$&2, 7, 30, 152, 878, 5653, 39952, 306419&new\\
\hline
$1^a2^a$&2, 7, 27, 114, 523, 2589, 13744, 77821& A011965\\
\hline
$1^a2^b$&2, 7, 26, 102, 426, 1909, 9210, 47787& new\\
\hline
\end{tabular}
\caption{Partition Patterns of Length 2}
\label{Tlength2}
\end{center}
\end{table}

We mention here a few observations about the asymptotic growth of the sequences given in Table \ref{Tlength2}.  We have seen that sequence $\left|\Pi_n \wr C_2(1^a2^a)\right|$, $n \geq 1$ is given by the second differences of the Bell numbers, that is $B(n+2)-2B(n+1)+B(n)$.   Note however that $B(n+2)-2B(n+1)+B(n)=B(n+2)\left(1-2\frac{B(n+1)}{B(n+2)}+\frac{B(n)}{B(n+2)}\right)$, where these last two terms vanish as $n \to \infty$.  That is, this sequence has the same asymptotic growth as the Bell numbers.  There are several results regarding the asymptotic growth of the Bell numbers.  In particular

$$B(n) \sim n^{-\frac{1}{2}}[\lambda(n)]^{n+\frac{1}{2}}e^{\lambda(n)-n-1}$$

where $\lambda(n)=\frac{n}{W(n)}$ and $W(n)$ is the Lambert W-function \cite{Lovasz93}. Or,

$$B(n) \sim \frac{n!}{\sqrt{2\pi W^2(n)e^{W(n)}}}\frac{e^{e^{W(n)}-1}}{W^n(n)} \cite{Odlyzko95}.$$

The sequence $\left|\Pi_n \wr C_2(1^a1^a)\right|$, $n \geq 1$ is slightly more unwieldy, with no known generating function at present, but computational results suggest that $$\lim_{n \to \infty} \dfrac{\left|\Pi_n \wr C_2(1^a1^a)\right|}{\left|\Pi_n \wr C_2(1^a2^a)\right|} = \infty,$$ that is, this new sequence grows asymptotically more rapidly than the Bell numbers.

The complexity of Theorem \ref{theorem1a2b} makes computation of $\displaystyle{\lim_{n \to \infty} \dfrac{\left|\Pi_n \wr C_2(1^a2^a)\right|}{\left|\Pi_n \wr C_2(1^a2^b)\right|}}$ even more unwieldly, and experimental evidence is less conclusive than for the previous ratio.  Although it appears that $\left|\Pi_n \wr C_2(1^a2^a)\right| > \left|\Pi_n \wr C_2(1^a2^b)\right|$ for $n \geq 3$, the ratio $\dfrac{\left|\Pi_n \wr C_2(1^a2^a)\right|}{\left|\Pi_n \wr C_2(1^a2^b)\right|}$ increases at a decreasing rate as $n$ increases.  It remains an open question to explore the asymptotic growth of $\left|\Pi_n \wr C_2(1^a2^b)\right|$.

\section{A Few Interesting Bijections}

As with any combinatorial object, interesting results arise from asking whether other objects have the same enumeration. We have already seen one such example in Section 2 where $\left|\Pi_n \wr C_2(1^a2^a)\right|$ produces OEIS Sequence A011965 (the second differences of the Bell numbers).  In this section we consider pattern sets $R \subset \Pi_2 \wr C_2$ such that the members of $\Pi_n \wr C_2(R)$ are in bijection with other well known combinatorial objects.  A summary of the results of this section can be found in Table \ref{Tlength2set}.

In Theorem \ref{theorem2ndbells} we computed $\left|\Pi_n \wr C_2 (1^a2^a)\right|$ by elementary counting methods.  Now we consider the relationship between the members of $\Pi_n \wr C_2 (1^a2^a)$ and a different set of partitions.

\begin{thm} There is a constructive bijection between $\Pi^{eq}_n\wr C_2(1^a2^a)$ and the partitions of $[n+3]$ with at least one singleton such that the largest singleton is n+1.
\end{thm}

\begin{proof} Here we consider partitions where all $a$-colored elements must appear in the same block, but there are no restrictions on where $b$-colored elements appear.

The partitions of $[n+3]$ with at least one singleton element such that the largest singleton is $n+1$ can be viewed as partitions of $[n]$ with two ``markers''.  These markers (corresponding to $n+2$ and $n+3$), may be (a) together in their own block, (b) inserted together in a non-empty block, or (c) inserted into distinct other non-empty blocks.  From this description, we see that yet another formula for this sequence is $\sum_{k=1}^n \stirr{n}{k}(k^2+1)$.

We find a bijection between our colored partitions and these ``marked'' partitions in the following way:

If the markers are together in their own block, the partition of $[n]$ corresponds to a colored partition of $[n]$ where all elements are colored with color $b$.

If the markers are together in a block with other elements, then that block consists only of $a$-colored elements, while all other elements are colored $b$.

If the markers are in distinct blocks with other elements, then the elements in the block with the first marker are $a$-colored elements, and the elements in the block with the second marker are $b$-colored elements, and these two blocks will be combined to yield a new block with some $a$-colored and some $b$-colored elements.  \end{proof}

We can break this argument into cases and show that our partitions are counted by $$B(n)+1+\sum_{k=1}^{n-1}\sum_{\ell=1}^{n-k}\binom{n}{k}\stirr{n-k}{\ell}(\ell+1),$$ and then algebraically manipulate to obtain $B(n+2)-2B(n+1)+B(n)$.  We can also use the formula $B(n+2)-2B(n+1)+B(n)$ to find an exponential generating function counting these colored partitions.

If $B(x)=\sum B(n) \dfrac{x^n}{n!}$ is the exponential generating function for the Bell numbers, then $\sum (B(n+1)-B(n)) \dfrac{x^n}{n!} = \sum B(n+1) \dfrac{x^n}{n!}-\sum B(n) \dfrac{x^n}{n!}=B^{\prime}(x)-B(x)$ is the exponential generating function for the first differences of the Bell numbers.  By similar argument, $(B^{\prime}(x)-B(x))^{\prime} - (B^{\prime}(x)-B(x) = B^{\prime \prime}(x)-2B^{\prime}(x)+B(x)$ is the exponential generating function for their second differences.  We know that the exponential generating function for the Bell numbers is $e^{e^x-1}$, so plugging this into the formula $B^{\prime \prime}(x)-2B^{\prime}(x)+B(x)$, we obtain that the exponential generating function for our partitions is $e^{e^x-1}(e^{2x}-e^x+1)$.

If we consider partitions which avoid a \emph{set} of patterns instead of just a single pattern, a number of the standard combinatorial sequences are produced.  As we see in Table \ref{Tlength2set}, the sets of patterns in the remainder of this section are each equivalent to a single pattern in the $lt$ sense or in the $pattern$ sense.

\begin{thm}
$$\left|\Pi_n^{eq} \wr C_2 (1^a1^a,1^a1^b,1^b1^a,1^b1^b)\right|=\left|\Pi_n^{eq} \wr C_2 (1^a2^a,1^a2^b,1^b2^a,1^b2^b)\right|=2^n \text{ (OEIS A000079)}.$$
\end{thm}

\begin{proof} A partition avoids the set of patterns $\{1^a1^a,1^a1^b,1^b1^a,1^b1^b\}$ exactly when each element must appear in its own block, and there are $2$ ways to color each element, yielding $2^n$  such partitions.

A partition avoids the set of patterns $\{1^a2^a,1^a2^b,1^b2^a,1^b2^b\}$ exactly when each element appears in the same block, and there are $2$ ways to color each element, yielding $2^n$ such partitions. \end{proof}

Such partitions can clearly be put into bijection with any set enumerated by $2^n$.  Similarly, we can construct sets of patterns in $\Pi_n \wr C_c$ counted by $c^n$ if we place restrictions on all ways two elements in the same block can be colored (as in the first set of patterns above) or if we place restrictions on all ways two elements across distinct blocks can be colored (as in the second set of patterns above).

\begin{thm}
$$\left|\Pi_n^{eq} \wr C_2 (1^a2^a,1^b2^b)\right|=2(2^n-1)=2^{n+1}-2 \text{ (OEIS A000918)}.$$
\end{thm}

\begin{proof} A partition avoids these patterns exactly when all $a$-colored elements are in the same block and all $b$-colored elements are in the same block.  For each of the $n$ elements, we may choose the color $a$ or $b$.  Then, so long as there are elements of both colors we may decide to put the set of $a$-colored elements into the same block as the $b$-colored elements, or leave them as two separate monochromatic blocks.  If all elements are already monochromatic (this can happen in two ways), then we are done.  This yields $2(2^n-2)+2 = 2(2^n-1)$ such partitions.\end{proof}

This set is clearly in bijection with the non-empty proper subsets of an $(n+1)$-element set.  Let the $a$-colored elements be the chosen elements less than $n+1$ in an $(n+1)$-element subset.  Then, if these elements are merged with the $b$-colored elements, add $n+1$ to the chosen subset.  Notice that we would never merge with $b$-colored elements if we had already chosen $\{1,\dots,n\}$ to be in our chosen subset, so this is indeed a proper subset of $[n+1]$.

\begin{thm}
$$\left|\Pi_n^{eq} \wr C_2 (1^a1^b,1^a1^a)\right|=\left|\Pi_n^{eq} \wr C_2 (1^b1^a,1^a1^a)\right|=\sum_{k=1}^n 2^k\stirr{n}{k} \text{ (OEIS A001861)}.$$
\end{thm}

\begin{proof} A partition avoids these patterns exactly when there is at most one $a$-colored element in a block and any such element is the last element of its block.  Thus, we may partition the elements of $[n]$ into $k$ blocks in $\stirr{n}{k}$ ways.  Then, we have 2 choices for the color of the last element of each block (all other elements must have color $b$).\end{proof}

This sequence appears in the OEIS as ``values of Bell polynomials: ways to place $n$ unlabeled balls in $n$ unlabeled but $2$-colored boxes'', which is exactly what we have just done -- partition our $n$ integers into any number of sets, but then $2$-color the (last elements of the) sets.

Using standard techniques the exponential generating function for this sequence is $e^{2e^x-2}-1$.

Note that the solution to American Mathematical Monthly Problem 11567 \cite{Monthly11}, which is $\sum_{k=1}^n 2^{n-k}\stirr{n}{k}$, bears a superficial resemblance to the enumeration above.  It is not hard to see that this enumeration counts the subset of $\Pi_n\wr C_2$ where the smallest element in each block must have color $a$.  There does not appear to be a way to describe this set using pattern avoidance.  

The final theorem of this section gives rise to a bijection between a certain set of pattern-avoiding colored partitions and a set of restricted involutions.

\begin{thm}
$$\left|\Pi_n^{eq} \wr C_2 (1^a1^a,1^b1^b)\right|=a(n) \text{where $a(n)$ satisfies }$$ $$a(n)=2(a(n-1)+(n-1)a(n-2)) \text{ (OEIS A000898)}.$$ 
\end{thm}

\begin{proof} A partition avoids this set of patterns exactly when there are no two elements of the same color in the same block.  To see why $a(n)=2(a(n-1)+(n-1)a(n-2))$, consider the element $n$.  Either $n$ is in a block with one of the other $(n-1)$ smaller elements, or $n$ is in a block by itself.  In the first case, there are $(n-1)$ ways to choose the element to go with $n$, two ways to color that block of size 2 so that each element has a different color, and $a(n-2)$ ways to partition and color the remaining elements.  In the second case, there are 2 ways to color $n$, and there are $a(n-1)$ ways to partition and color the other elements, yielding the above recurrence. \end{proof}

This sequence also counts $I_{2n}^{rc}$, the number of involutions of length $2n$ that are invariant under the reverse-complement map \cite{EggeSymm}.  The {\it reversal map} for permutations is the map that sends $\pi=\pi_1\pi_2\dots\pi_n$ to $\pi_n\pi_{n-1}\dots\pi_1$, and the {\it complement map} for permutations is the map that sends $\pi=\pi_1\pi_2\dots\pi_n$ to $(n+1-\pi_1)(n+1-\pi_2)\dots(n+1-\pi_n)$.  These maps are bijections, which commute with each other, and we refer to their composition as the \emph{reverse-complement map}.

We can obtain a bijection between the partitions avoiding $\{1^a1^a,1^b1^b\}$ and such involutions in the following way:

Consider an involution $\pi$ of length $2n$ that is invariant under the reverse-complement map, and let $p$ be the position of $2n$.  Either $p=1$, $p=2n$ or $2 \leq p \leq 2n-1$.  We consider each of these cases in turn.

If $p=1$, then since $\pi$ is invariant under the reverse-complement map, we have that $\pi(1)=2n$ and $\pi(2n)=1$.  The other $2n-2=2(n-1)$ entries of $\pi$ must also form an involution invariant under the reverse-complement map, i.e. there are $a(n-1)$ ways to choose the remaining entries of the involution.  Let such an involution correspond to a colored partition where $n$ is in a block by itself and $n$ is colored with color $a$.

Similarly, if $p=2n$, then $\pi(2n)=2n$ and $\pi(1)=1$.  Again there are $a(n-1)$ ways to choose the order of the remaining entries of the involution.  We will let such an involution correspond to a colored partition where $n$ is in a block by itself and $n$ is colored with color $b$.

Now, we consider the case where $2 \leq p \leq 2n-1$.  Given such a value for $p$, the following are true:
\begin{itemize}
\item $\pi(p)=2n$
\item $\pi(2n)=p$
\item $\pi(2n+1-p)=1$
\item $\pi(1)=2n+1-p$
\item $p$ can be written uniquely as either $2n-i$ or $1+i$ for some value $1 \leq i \leq n-1$
\end{itemize}

We note immediately that there are only $2n-4 = 2(n-2)$ remaining entries of the involution to arrange.  If $p=1+i$ for some $1 \leq i \leq n-1$, then the corresponding colored partition will have $n$ in the same block as $i$ where $n$ is colored with color $a$ and $i$ is colored with color $b$.  If $p=2n-i$ for some $1 \leq i \leq n-1$ then the corresponding colored partition will have $n$ in the same block as $i$ where $n$ is colored with color $b$ and $i$ is colored with color $a$.

This bijection addresses what happens in the block involving $n$ and works recursively to handle all other blocks of the colored partition.

Again using standard techniques we compute the exponential generating function for this sequence to be $e^{x^2+2x}$, which gives the formula $\sum_{k} \binom{n}{2k} \binom{2k}{k} k! 2^{n-2k}$. 

\begin{table}[hbt]
\begin{center}
\begin{tabular}{|c|c|c|c|c|}
\hline
Pattern Set $R$&Equivalent Pattern&OEIS&Formula for $\left|\Pi_n \wr C_2(R)\right|$\\
& in $\Pi_2 \wr C_2$&number&\\
\hline
$\{1^a2^a\}$&& A011965& $B(n+2)-2B(n+1)+B(n)$\\
\hline
$\{1^a1^a,1^a1^b,1^b1^a,1^b1^b\}$&$lt$-$1^22^2$&A000079& $2^n$\\
\hline
$\{1^a2^a,1^b2^b\}$&$pattern$-$1^a2^a$& A000918& $2^{n+1}-2$\\
\hline
$\{1^a1^a,1^a1^b\}$&$lt$-$1^12^2$& A001861& $\sum_{k=1}^n 2^k\stirr{n}{k}$\\

\hline
$\{1^a1^a,1^b1^b\}$&$pattern$-$1^a1^a$& A000898& $\sum_{k} \binom{n}{2k} \binom{2k}{k} k! 2^{n-2k}$\\
\hline
\end{tabular}
\caption{Sets of Partition Patterns in $\Pi_2 \wr C_2$}
\label{Tlength2set}
\end{center}
\end{table}

\section{Wilf Classes for Colored Patterns of Length 3}

We have completely enumerated partitions which avoid a single pattern in $\Pi_2 \wr C_c$, and partitions which avoid certain sets of length 2 patterns.  We continue our classification with length 3 patterns.  There are 25 patterns in $\Pi_3 \wr C_c$ with an arbitrary number of colors in $eq$ avoidance.
They are:
\begin{center}
\begin{tabular}{ccccc}
$1^a1^a1^a$  &$1^a1^a1^b$  &$1^a1^b1^a$ &$1^b1^a1^a$  &$1^a1^b1^d$\\
$1^a1^a2^a$  &$1^a1^a2^b$  &$1^a1^b2^a$ &$1^b1^a2^a$  &$1^a1^b2^d$\\
$1^a2^a1^a$  &$1^a2^a1^b$  &$1^a2^b1^a$ &$1^b2^a1^a$  &$1^a2^b1^d$\\
$1^a2^a2^a$  &$1^a2^a2^b$  &$1^a2^b2^a$ &$1^b2^a2^a$  &$1^a2^b2^d$\\
$1^a2^a3^a$  &$1^a2^a3^b$  &$1^a2^b3^a$ &$1^b2^a3^a$  &$1^a2^b3^d$\\
\end{tabular}
\end{center}

In this section, we show several equivalences.  In the following section we enumerate partitions avoiding these patterns.

Given a colored partition $P=p_1^{c_1}\cdots p_n^{c_n}$, define the \emph{reversal} of $P$ as $Q=q_1^{c_n} \cdots q_n^{c_1}$ where $q_1 \cdots q_n$ is the canonization of $p_n \cdots p_1$.  Then it is clear that partitions avoiding $P$ are equinumerous with partitions avoiding $Q$.

For this reason we have

$$\left|\Pi_n^{eq} \wr C_c(1^a1^a1^b)\right|=\left|\Pi_n^{eq} \wr C_c(1^b1^a1^a)\right|$$

$$\left|\Pi_n^{eq} \wr C_c(1^a2^a3^b)\right|=\left|\Pi_n^{eq} \wr C_c(1^b2^a3^a)\right|$$

$$\left|\Pi_n^{eq} \wr C_c(1^b2^a2^a)\right|=\left|\Pi_n^{eq} \wr C_c(1^a1^a2^b)\right|$$

$$\left|\Pi_n^{eq} \wr C_c(1^a2^b2^d)\right|=\left|\Pi_n^{eq} \wr C_c(1^a1^b2^d)\right|$$

$$\left|\Pi_n^{eq} \wr C_c(1^a2^b2^a)\right|=\left|\Pi_n^{eq} \wr C_c(1^a1^b2^a)\right|$$

$$\left|\Pi_n^{eq} \wr C_c(1^a2^a2^a)\right|=\left|\Pi_n^{eq} \wr C_c(1^a1^a2^a)\right|$$

$$\left|\Pi_n^{eq} \wr C_c(1^a2^a2^b)\right|=\left|\Pi_n^{eq} \wr C_c(1^b1^a2^a)\right|$$

$$\left|\Pi_n^{eq} \wr C_c(1^a2^a1^b)\right|=\left|\Pi_n^{eq} \wr C_c(1^b2^a1^a)\right|$$

Notice, in all equations above, since we are dealing with $eq$-avoidance, the order of the colors does not matter, just whether particular colors are different or the same.

These equivalences narrow the work on patterns of length three from 25 cases to 17 cases.

As noted in the previous section, if $c$ is less than the number of colors in a forbidden pattern $P$, then $\left|\Pi_n^{eq} \wr C_c(P)\right|=c^nB(n)$.  That is, there are not enough colors in $\Pi_n \wr C_c$ to contain a copy of the forbidden pattern $P$, so \emph{all} members of $\Pi_n \wr C_c$ avoid $P$.  But if $c$ is sufficiently large, we observe the following additional equivalences.

$$\left|\Pi_n^{eq} \wr C_c(1^a1^a1^b)\right|=\left|\Pi_n^{eq} \wr C_c(1^a1^b1^a)\right|$$

To show that the avoidance sets are equinumerous it is sufficient to show that the number of colored partitions that \emph{contain} each of these patterns is equinumerous.

Let $P$ be a partition that contains $1^a1^a1^b$, but not $1^a1^b1^a$.  Then, there is at least one block with two elements colored $a$.  For each block $B_i=\{a_1^{c_1},a_2^{c_2},\dots,a_\ell^{c_\ell}\}$ with at least two $a$-colored elements, suppose that $j$ is the smallest index such that $c_j=a$.  Recolor $B_i$ as follows.  $$c_k=\left\{\begin{array}{ll}c_k&\mbox{if }k\leq j\\ c_{\ell-k+j+1}& \mbox{if }j+1\leq k\leq \ell.\end{array}\right.$$  Leave the remaining blocks alone.  This construction is its own inverse.

Since the set of $1^a1^a1^b$-containing partitions is equinumerous with the set of $1^a1^b1^a$-containing partitions, their avoidance sets must also be equinumerous.  

$$\left|\Pi_n^{eq} \wr C_c(1^a1^a1^a)\right|=\left|\Pi_n^{eq} \wr C_c(1^a1^a1^b)\right|$$

Let $P$ be a partition that contains $1^a1^a1^a$ but not $1^a1^a1^b$.  This means that there is at least one block with at least three elements colored $a$.  For each block with three or more $a$-colored elements, recolor all but the first two $a$-colored elements with color $b$.  We have just created a partition that contains $1^a1^a1^b$ but not $1^a1^a1^a$.  This process is easily reversed.

$$\left|\Pi_n^{eq} \wr C_c(1^a1^b1^a)\right|=\left|\Pi_n^{eq} \wr C_c(1^a1^b1^d)\right|$$

Let $P$ be a partition that contains $1^a1^b1^a$ but not $1^a1^b1^d$.  This means that there is at least one block with three elements colored $a$, $b$, and $a$ respectively.  For each such block, let $i$ be the first element colored $a$, and let $j$ be the least element colored $b$ such that $j>i$.  Finally, locate all $a$-colored elements $k$ such that $k>j$, and recolor them with color $d$.  We have now produced a partition that contains $1^a1^b1^d$ but not $1^a1^a1^b$.  This process is easily reversed.

$$\left|\Pi_n^{eq} \wr C_c(1^a2^b1^a)\right|=\left|\Pi_n^{eq} \wr C_c(1^a2^b1^d)\right|$$

Let $P$ be a partition that contains $1^a2^b1^a$ but not $1^a2^b1^d$.  This means that there is at least one block with at least two elements colored $a$.  Within each block, let all $a$-colored elements except the smallest be re-colored with color $d$.  Now, if $i$ and $j$ were originally in the same block, participating in a $1^a2^b1^a$ pattern, the smallest $a$-colored element in the block along with $j$ are also the first and last elements in a (possibly different) $1^a2^b1^a$ pattern.  By changing $j$ to be color $d$, $P$ now contains a $1^a2^b1^d$ pattern, and we have removed all $1^a2^b1^a$ patterns since there are no longer any blocks with two $a$-colored elements.

On the other hand, let $Q$ be a partition that contains $1^a2^b1^d$ but not $1^a2^b1^a$.  This means there is at least one block with a smaller $a$-colored element followed by a larger $d$-colored element.  For all such blocks, take all the $d$-colored elements and change them to $a$-colored elements.  As before, this process removes all $1^a2^b1^d$ patterns, and guarantees the existence of a $1^a2^b1^a$ pattern.

Since the set of members of $\Pi_n \wr C_c$ that contain $1^a2^b1^a$ is equinumerous with the set of members of $\Pi_n \wr C_c$ that contain $1^a2^b1^d$, their avoidance sets must also be equinumerous.

$$\left|\Pi_n^{eq} \wr C_c(1^b2^a2^a)\right|=\left|\Pi_n^{eq} \wr C_c(1^a2^b2^d)\right|$$

This argument is parallel to the previous argument.  Since the particular numbers represented by $a$ and $b$ are insignificant, we can rewrite $1^b2^a2^a$ as $1^a2^b2^b$.  Now, in a partition that contains $1^a2^b2^b$ but avoids $1^a2^b2^d$, find all blocks with multiple $b$-colored elements and change all but the first $b$-colored element to have color $d$.  For a partition that contains $1^a2^b2^d$ but not $1^a2^b2^b$, find all blocks with a smaller $b$-colored element followed by a larger $d$ colored element.  In all such blocks recolor all $d$-colored elements after the first $b$-colored element to have color $b$ as well.

$$\left|\Pi_n^{eq} \wr C_c(1^a2^a2^a)\right|=\left|\Pi_n^{eq} \wr C_c(1^a2^a2^b)\right|$$

Again this argument is parallel to the previous two equivalences.  In a partition that contains $1^a2^a2^a$ but not $1^a2^a2^b$, find all blocks with more than one $a$-colored element.  In all such blocks, recolor all but the smallest $a$-colored element to have color $b$.  On the other hand, for a partition that contains $1^a2^a2^b$ but not $1^a2^a2^a$, find all blocks with a smaller $a$-colored element followed by larger $b$ colored elements and recolor all $b$-colored elements after the smallest $a$-colored element to have color $a$.

$$\left|\Pi_n^{eq} \wr C_c(1^a2^a2^b)\right|=\left|\Pi_n^{eq} \wr C_c(1^a2^a1^b)\right|$$

This final equivalence is somewhat different from the previous three.

Consider a partition that contains the pattern $1^a2^a1^b$ but avoids the pattern $1^a2^a2^b$.  Find all $a$-colored elements that participate in some instance of a $1^a2^a1^b$ pattern.  Order these elements $n_1\cdots n_\ell$ and replace $n_i$ with $n_{\ell+1-i}$.  Now every $1^a2^a1^b$ pattern has been transformed into a $1^a2^a2^b$ pattern.  Reversing this process changes all $1^a2^a2^b$ patterns to $1^a2^a1^b$ patterns.

These new equivalences further narrow our work to 10 cases. The 25 patterns in question are organized by equivalence class below.

\begin{itemize}
\item $1^a1^a1^a$, $1^a1^a1^b$, $1^a1^b1^a$, $1^b1^a1^a$, $1^a1^b1^d$
\item $1^a2^a2^a$, $1^a1^a2^a$, $1^a2^a2^b$, $1^b1^a2^a$, $1^a2^a1^b$, $1^b2^a1^a$
\item $1^a2^b2^a$, $1^a1^b2^a$
\item $1^b2^a2^a$, $1^a1^a2^b$, $1^a2^b2^d$, $1^a1^b2^d$
\item $1^a2^a1^a$
\item $1^a2^b1^a$, $1^a2^b1^d$
\item $1^a2^a3^a$
\item $1^a2^a3^b$, $1^b2^a3^a$
\item $1^a2^b3^a$
\item $1^a2^b3^d$
\end{itemize}

It turns out no further equivalences exist.  Each of these 10 classes produces a different counting sequence.  We will address each of them in turn.

\section{Colored Patterns of Length 3}\label{seclen3}

In this section we address the enumeration of colored partitions avoiding a single pattern of length 3.  Table \ref{Tlength3} summarizes the results of this section.  It should be noted that although some of the formulas in this section at first glance appear unwieldy, brute force computation of $\left|\Pi_n^{eq} \wr C_c (P)\right|$ is only feasible for $n \leq 9$ with $c=2$, and only for even smaller values of $n$ if $c>2$.  Each of the formulas presented here is easily programmable offering an exponential speedup in the computation of $\left|\Pi_n^{eq} \wr C_c (P)\right|$.  Thus, these truly are formulas in the Wilfian sense \cite{Wilf}.

As in Section 2, we note that $\left| \Pi_n^{eq} \wr C_c(P)\right|=c^nB(n)$ if $c$ is less than the number of colors in $P$ or if $n$ is less than the length of $P$, so again all theorems in this section hold for $n$ and $c$ sufficiently large.

\subsection{Patterns of the Form 111}

A partition $P \in \Pi_n\wr C_c$ avoids a partition of the form 111 if and only if each of its blocks avoid this pattern.  For this reason we will construct formulas that involve sums over integer partitions.  This will allow us to determine the number of elements in each block from the outset and construct the formula appropriately.  This section features the enumeration of the Wilf class $\{1^a1^a1^a\}$.  In this section we also generalize the formula for this class to the pattern $\underbrace{1^a1^a\dots1^a}_{m+1}$.

As we saw in the previous section, for $c\geq2$, $$\left|\Pi_n^{eq} \wr C_c ( 1^a1^a1^a)\right|=\left|\Pi_n^{eq} \wr C_c ( 1^a1^a1^b)\right|=\left|\Pi_n^{eq} \wr C_c ( 1^a1^b1^a)\right|=\left|\Pi_n^{eq} \wr C_c ( 1^b1^a1^a)\right|,$$ and for $c \geq 3$, $\left|\Pi_n^{eq} \wr C_c ( 1^a1^a1^a)\right|=\left|\Pi_n^{eq} \wr C_c ( 1^a1^b1^d)\right|$, so for $c$ sufficiently large, we produce a formula for $\left|\Pi_n^{eq} \wr C_c (P)\right|$ where $P$ is any pattern of the form $1^{c_1}1^{c_2}1^{c_3}$.  One can use similar arguments to the equivalences in Section 4, to show that for $c$ sufficiently large our generalization to the pattern $\underbrace{1^a1^a\dots1^a}_{m+1}$ works for any pattern of the form $\underbrace{1^{c_1}1^{c_2}\dots1^{c_{m+1}}}_{m+1}$ regardless of whether the $c_i$'s are the same or distinct.

Let $P=B_1/B_2/\dots/B_k$ be a partition of $[n]$.  Define the {\it block structure} of $P$ to be the partition of the integer $n$ with part sizes $|B_1|,|B_2|,\dots,|B_k|$.  Thus, we can build set partitions by choosing an integer partition to be the block structure and then choose elements to put in each block.  Let $p$ be any integer partition and let $\#p(i)$ be the number of occurrences of $i$ in $p$. 

\begin{thm}   

$\displaystyle{\left| \Pi_n^{eq} \wr C_c\left(\underbrace{1^a1^a\dots1^a}_{m+1}\right)\right|=}$

$$\sum_{(p_1,p_2,\dots,p_k)\vdash n}\binom{n}{p_1,p_2,\dots,p_k}\cdot \prod_{i=1}^n\frac{1}{(\# p(i))!}\cdot\prod_{i=1}^k\left(\sum_{\ell_i=0}^m\binom{p_i}{\ell_i}(c-1)^{p_i-\ell_i}\right).$$ \end{thm}

\begin{proof} To avoid the pattern $\underbrace{1^a1^a\dots1^a}_{m+1}$, we may not have any block with $m+1$ $a$-colored elements.  We first produce a partition with block structure given by the integer partition $(p_1,p_2,\dots,p_k)$.  In block $i$ we choose $\ell_i$ of the elements to be colored $a$ and color them as such.  The remaining $p_i-\ell_i$ elements may be colored with any of the $c-1$ remaining colors.  So there are $\binom{p_i}{\ell_i}(c-1)^{p_i-\ell_i}$ ways to color $\ell_i$ of the elements of block $i$ with the color $a$.  If we sum from $\ell_i=0$ to $m$ then the we get every possible coloring of block $i$.  Taking the product from $i=1$ to $k$ we obtain all possible colorings of the partitions of $[n]$ with block structure $(p_1,p_2,\dots,p_k)$ that avoid $\underbrace{1^a1^a\dots1^a}_{m+1}$.  

Summing over all possible partitions with the given block structure gives us the desired result.
\end{proof}

The following is a direct consequence of Theorem 5.1.  

\begin{cor}   For $n\geq 3$,

$$\left| \Pi_n^{eq} \wr C_c\left(1^a1^a1^a\right)\right|=\sum_{(p_1,p_2,\dots,p_k)\vdash[n]}\binom{n}{p_1,p_2,\dots,p_k}\cdot \prod_{i=1}^n\frac{1}{(\# p(i))!}\cdot\prod_{i=1}^k\sum_{\ell_i=0}^2\binom{p_i}{\ell_i}(c-1)^{p_i-\ell_i}.$$\end{cor}

\subsection{Remaining Length 3 Patterns}

\begin{thm}
$$\left|\Pi_n^{eq} \wr C_c (1^a2^a3^a)\right|=B(n)(c-1)^n+\sum_{i=1}^n\sum_{r=1}^i \binom{n}{i}\binom{i}{r}B(n-i)(c-1)^{n-r}$$
$$+\sum_{i=2}^n\sum_{j=1}^{i-1}\sum_{r_1=1}^{i-j}\sum_{r_2=1}^j \binom{n}{i}\binom{i-1}{j}\binom{i-j}{r_1}\binom{j}{r_2}B(n-i)(c-1)^{n-r_1-r_2}.$$
\end{thm}

\begin{proof}
Notice that a partition avoids $1^a2^a3^a$ exactly when at most two blocks of the partition that have $a$-colored elements.  We will break our argument into 3 cases: (1) there are no $a$-colored elements in partition $p$, (2) there is precisely one block of partition $p$ that contains $a$-colored elements, and (3) there are precisely two blocks of partition $p$ that contain $a$-colored elements.

In case 1 we simply partition all $n$ elements into any number of blocks and color them with one of the $c-1$ colors that are not $a$.  There are $B(n)(c-1)^n$ ways to do this.

In case 2 we first pick $i \geq 1$ elements to be in the block with $a$-colored letters, then pick $r \geq 1$ of those $i$ elements to have color $a$.  Then, as in case 1, partition the remaining $n-i$ elements into some number of blocks and color the $n-r$ non-$a$-colored elements with one of the $c-1$ colors that are not $a$.  Summing over all reasonable values of $i$ and $r$ gives the double summation in the formula.

In case 3 we first pick $i \geq 2$ elements to be in the blocks with $a$-colored elements and then pick $j$ of them to be in the block not containing the smallest chosen element.  Now, pick $r_1$ of the $i-j$ elements in the one block and $r_2$ of the $j$ elements in the other block to be colored with color $a$.  Finally partition the remaining $n-i$ elements and color all $n-r_1-r_2$ elements that will not have color $a$.  Summing over all reasonable values of $i$, $j$, $r_1$, and $r_2$ gives the quadruple summation in the theorem.
\end{proof}

\begin{thm}
$$\left|\Pi_n^{eq} \wr C_c (1^a2^a1^a)\right|=B(n)(c-1)^n+\sum_{i=1}^n\sum_{r=1}^i\binom{n}{i}\binom{i}{r}B(n-i)(c-1)^{n-r}$$
$$+\sum_{i=2}^n\sum_{\ell=2}^i\sum_{j=0}^{n-i}\binom{n}{i}\binom{i-1}{l-1}\binom{n-i}{j}\ell^jB(n-i-j)(c-1)^{n-i}.$$
\end{thm}

\begin{proof}
Notice that we can avoid the pattern $1^a2^a1^a$ in 3 ways: (1) there are no $a$-colored elements, (2) all the $a$-colored elements are in one block, or  (3) there are at least two blocks with $a$-colored elements and there is an ordering of the blocks of the partition that puts all $a$-colored elements in increasing order.

Cases 1 and 2 were accounted for in the previous theorem, so we will focus on case 3.  First, choose $i \geq 2$ elements to be colored with color $a$.  Now, write those $i$ elements in increasing order.  There are $\binom{i-1}{\ell-1}$ ways to partition these $i$ elements into $\ell$ non-empty blocks so that the $i$ elements stay in order.  Now, choose $j\geq 0$ elements to be placed in the $\ell$ blocks with these $i$ $a$-colored elements.  There are $\binom{n-i}{j}$ ways to choose these elements and $\ell^j$ ways to assign each of the $j$ elements to one of the $\ell$ blocks.  Finally, partition the remaining $n-i-j$ elements into some number of blocks and color all $n-i$ non-$a$ colored elements with one of $c-1$ colors that are not $a$.  Summing over reasonable values of $i$, $\ell$, and $j$ gives the theorem.
\end{proof}

\begin{thm}
$$\left|\Pi_n^{eq} \wr C_c (1^a2^a2^a)\right| = B(n)(c-1)^n+B(n)n(c-1)^{n-1}$$
$$+\sum_{i=2}^n\sum_{\ell=1}^i\sum_{j=0}^{n-i}\binom{n}{i}\binom{n-i}{j}(i-\ell+1)^jB(n-i-j)(c-1)^{n-i}$$
$$+\sum_{i=2}^n\sum_{\ell=1}^{i-2}\sum_{j=0}^{n-i}\binom{n}{i}\binom{n-i}{j}(i-\ell-1)(i-\ell)^jB(n-i-j)(c-1)^{n-i}.$$
\end{thm}
\begin{proof}
Let $B$ be the block with the smallest $a$-colored element.  Avoiding the pattern $1^a2^a2^a$ means that (i) any block other than $B$ may have at most one $a$-colored element, and (ii) if $j$ is the smallest $a$-colored element not in $B$, then $B$ may have at most one $a$-colored element larger than $j$.

As usual, we will proceed by cases.  Either (1) there are no $a$-colored elements, (2) there is exactly one $a$-colored element, (3) there is more than one $a$-colored element, but block $B$ has no $a$-colored element that is larger than an $a$-colored element from a different block, or (4) there is more than one $a$-colored element, and block $B$ does have one $a$-colored element that is larger than an $a$-colored element from a different block.

As always, case 1 can be accomplished in $B(n)(c-1)^n$ ways by partitioning the $n$ elements into blocks and then coloring each element with one of the $c-1$ non-$a$ colors.

In case 2, we may partition the $n$ elements in $B(n)$ ways, choose one element to be colored with color $a$ in $n$ ways, and then color the remaining $n-1$ elements with colors other than $a$ in $(c-1)^{n-1}$ ways.

In case 3, first choose $i \geq 2$ elements to be colored $a$, and choose $j \geq 0$ non-$a$-colored elements to go in the same blocks as these original $i$.  Let the smallest $\ell$ of the $a$-colored elements be in block $B$.  Then the other $i-\ell$ $a$-colored elements must each be in a distinct block.  Thus there are $(i-\ell+1)$ blocks that contain $a$-colored elements, so there are $(i-\ell+1)^j$ ways to assign the chosen $j$ non-$a$-colored elements to blocks.  As usual, there are $B(n-i-j)$ ways to partition the remaining elements into new blocks and $(c-1)^{n-i}$ ways to color the $n-i$ elements that do not have color $a$.

In case 4, we proceed similarly to case 3, choose $i$ elements to have color $a$, $j$ elements that do not have color $a$ to be in the same blocks as the original $i$, and let the smallest $\ell$ of the $a$-colored elements be in block $B$.  As in case 3, this guarantees that there will be $1+i-\ell$ blocks with $a$-colored elements if all the remaining $i-\ell$ $a$-colored elements are in their own block.  However, in case 4, we assume that there is one larger $a$-colored element that is also in block $B$. If $\ell \leq i-2$, we have at least 3 blocks with $a$-colored elements initially, and now we need to merge one of the $i-\ell-1$ largest $a$-colored elements into block $B$.  There are $(i-\ell-1)$ ways to choose a larger $a$-colored element to do this to, and then which gives $i-\ell$ blocks with $a$-colored elements.  There are $(i-\ell)^j$ ways to assign the $j$ non-$a$-colored elements to blocks with $a$ colored elements, then partition the remaining $n-i-j$ elements into blocks and color all $n-i$ non-$a$ colored elements in $(c-1)^{n-i}$ ways.
\end{proof}

For the remaining theorems we will use Greek letters for colors, to make keeping track of variables easier.

\begin{thm} 

$$\left| \Pi_n^{eq} \wr C_c(1^\alpha2^\alpha3^\beta)\right|=B(n)(c-1)^n+\sum_{i=1}^n\binom{n}{i}\sum_{j=0}^{n-i}\binom{n-i}{j}B(n-i-j)(c-1)^{n-i} + $$ $$\sum_{1\leq i<j\leq n}\sum_{a,b}\sum_{d,e}\sum_{f,g}\binom{i-1}{a,b}\binom{j-i-1}{d,e}\binom{n-j}{f,g}\cdot$$ $$B(n-a-b-d-e-f-g-2)\cdot (c-2)^{n-j-f-g}(c-1)^{j-d-2}c^{d+f+g}+$$  $$\sum_{1\leq i<j<k\leq n}\sum_{a,b}\sum_{d,e}\sum_{f,g}\sum_{\ell,m}\sum_p\binom{i-1}{a,b}\binom{j-i-1}{d,e}\cdot$$ $$\binom{k-j-1}{f,g}\binom{n-k}{\ell,m}\binom{n-a-b-d-e-f-g-\ell-m-3}{p}\cdot$$ $$B(n-a-b-d-e-f-g-\ell-m-p-3)(c-2)^{k-j-f-g-1}(c-1)^{n-k-d+j-2}c^{d+f+g}.$$\end{thm}

\begin{proof}  Let $P\in\Pi_n\wr C_c$ avoid $1^\alpha2^\alpha3^\beta$.  As usual we break this into cases.  The first case is that $P$ avoids the pattern $1^\alpha2^\alpha$.  This was done previously.  Case two is that $P$ contains a copy of $1^\alpha2^\alpha$.  In the second case, we condition on the location of the first copy of $1^\alpha2^\alpha$ and we will break this case into two subcases depending on if there is an element larger than these two colored $\alpha$ and in a different block altogether .  

Suppose that $i$ and $j$ with $i<j$ are the locations of the $1^\alpha$ and $2^\alpha$ respectively.  Place $i$ in block $B$ and place $j$ in block $C$.

Case 2.1:  Suppose that no elements outside of blocks $B$ and $C$ are colored $\alpha$.  Then we construct such a partition in the following way.  Of the first $i-1$ elements choose $a$ elements to be in block $B$ and $b$ elements to be in block $C$.  Now, of the $j-i-1$ elements between $i$ and $j$ choose $d$ to be in block $B$ and $e$ to be in $C$.  Now, of the $n-j$ elements larger than $j$ place $f$ in block $B$ and $g$ in block $C$.  Now partition the remaining $n-a-b-d-e-f-g-2$ elements.  This can be done in $B(n-a-b-d-e-f-g-2)$ ways.

Now, we need to color these elements.  Both $i$ and $j$ must be colored $a$.  The first $i-1$ elements can be colored in $c-1$ ways since if one of them is colored $a$ we have an earlier occurrence of $1^\alpha2^\alpha$.  The $d$ elements in block $B$ between $i$ and $j$ can each be colored in $c$ ways.  The remaining $j-i-d-1$ elements must be colored anything but $\alpha$ again.  The $f+g$ elements after $j$ in block $B$ or $C$ can each be colored in $c$ ways.  The remaining $n-j-f-g$ elements cannot be colored $\alpha$ because no elements outside of $B$ and $C$ can be colored $\alpha$.  Also, none of these elements can be colored $\beta$ otherwise we would have a copy of $1^\alpha2^\alpha3^\beta$, so they can each be colored in $c-2$ ways.  

Case 2.2:  Suppose that there is some element outside of blocks $B$ and $C$ colored $\alpha$.  As was described in the coloring of the partition above, this element must be some $k>j$.  Let $k$ be the smallest such element.  We form a partition in the same way we did above except that we choose $f$ elements for block $B$ and $g$ elements for block $C$ from those between $j$ and $k$ and we choose $\ell$ elements for block $B$ and $m$ elements for block $C$ from those greater than $k$.  Now, of the remaining $n-a-b-d-e-f-g-\ell-m-3$ choose $p$ to be in the block with $k$.  Finally, partition the remaining elements.   

Here, the elements $i$, $j$, and $k$ must be colored $\alpha$.  The elements less than $k-1$ are colored exactly as they were in the previous case.  The elements greater than $k$ cannot be colored $\beta$, so there are $(c-1)^{n-k}$ ways to color these.

\end{proof}

\begin{thm} 

$\displaystyle{\left|\Pi_n^{eq} \wr C_c(1^\alpha1^\alpha2^\beta) \right|=}$

$$B(n)(c-1)^n+\sum_i\sum_a\sum_b\sum_d\sum_{e=0}^{\left\lfloor(n-a-b-d-1)/2\right\rfloor}\sum_p \binom{i-1}{a}\binom{n-i}{b}\binom{i-a-1}{d,e}\cdot$$ $$\binom{n-a-b-d-e-1}{p}\stirr{p}{e}\cdot e!\cdot B(n-a-b-d-e-p-1)c^\alpha(c-1)^{n-i}(c-2)^{i-a-d-e-1}+$$ $$\sum_{1\leq i<j\leq n}\sum_{a,b}\sum_{d,e}\sum_{f,g}\sum_{h}\sum_{k=0}^{\left\lfloor(n-a-b-d-e-f-g-h)/2\right\rfloor}\sum_{\ell}\sum_{m=0}^{\left\lfloor(n-a-b-d-e-f-g-h-k-\ell)/2\right\rfloor}\sum_{p\geq k+m}\binom{i-1}{a,b}\cdot$$ $$\binom{j-i-1}{d,e}\binom{n-j}{f,g}\binom{i-a-b-1}{h,k}\binom{j-i-d-e}{\ell,m}\cdot$$ $$\binom{n-a-b-d-e-f-g-h-k-\ell-m-2}{p}\cdot \stirr{p}{k+m}\cdot (k+m)!\cdot$$ $$ B(n-a-b-d-e-f-g-h-k-\ell-m-p-2)\cdot$$ $$(a(c-1)^{a-1}(c-2)^d+d(c-1)^a(c-2)^{d-1}+(c-1)^a(c-2)^d)(b(c-1)^{b-1}c^e+(c-1)^b c^e)\cdot$$ $$(c-1)^{n-j+i-a-b-h-k-1}(c-2)^{j-i-d-e-\ell-m-1}.$$\end{thm}

\begin{proof} We have three cases.  

In the first case if no element is colored $\beta$ then there are $B(n)$ way to partition the elements and $(c-1)^n$ ways to color them.  

Suppose now that exactly one block contains elements colored $\beta$.  Call this block $B$ and let $i$ be the largest element in $B$ colored $\beta$.  Now, choose $a$ elements less than $i$ and $b$ elements larger than $i$ to to be in $B$.  Pick $d+e$ elements less than $i$ to be colored $\alpha$.  Let $d$ of them be singletons and $e$ of them go into blocks with other elements.  Since the partition must avoid copies of $1^\alpha1^\alpha2^\beta$, we cannot put any of the $e$ elements colored $\alpha$ in the same block.  Choose $p\geq d$ elements from the remaining elements to go into blocks with the $d$ elements colored $\alpha$.  Partition these into $e$ blocks and place the $e$ elements into these blocks.  

Now, we color the partitions.  Notice that the elements larger than $i$ cannot be colored $\beta$, but can be colored with any other color, so there are $(c-1)^{n-i}$ ways to color these elements.  Of the elements less than $i$, those in $B$ may receive any color and the remaining elements may not be colored $\alpha$ or $\beta$.

The final case is that there are at least two blocks with elements colored $\beta$.  Let $j$ be the largest element colored $\beta$ and let $j$ be in block $B$.  Let $i$ be the largest element not in $B$ colored $\beta$ and let $i$ be in block $A$.  Pick $a$ elements, $d$ elements, and $f$ elements to go in block $A$ from the elements less than $i$, between $i$ and $j$ and larger than $j$ respectively.  Now, choose $b$ elements, $e$ elements, and $g$ elements to go in block $B$ from the elements less than $i$, between $i$ and $j$ and larger than $j$ respectively. From the remaining elements less than $i$, choose $h$ elements to be colored $\alpha$ and go in singleton blocks and $k$ elements to be colored $\alpha$ and go into blocks with other elements.  From the remaining elements between $i$ and $j$, choose $\ell$ elements to be colored $\alpha$ and go in singleton blocks and $m$ elements to be colored $\alpha$ and go into blocks with other elements.  This will prevent a copy of $1^\alpha1^\alpha$ from appearing before an element colored $\beta$ in a different block.  Now, choose $p$ elements from the remaining elements to join the $k+m$ elements colored $\alpha$ that are to be in blocks with other elements.  Since none of these elements can be in a block together we partition the $p$ elements in $k+m$ blocks and distribute the $k+m$ elements colored $\alpha$ into these blocks.  This can be done in $\stirr{p}{k+m}\cdot(k+m)!$ ways.  Now partition the remaining elements.  

We have already colored $h+k+\ell+m+2$ elements.  At most one element from $A$ less than $j$ can be colored $\alpha$, so depending on the location and existence of this element the number of ways to color the elements in $A$ that are less than $j$ is $\alpha(c-1)^{\alpha-1}(c-2)^d+d(c-1)^a(c-2)^{d-1}+(c-1)^a(c-2)^d$ since no elements in $A$ between $i$ and $j$ can be colored $\beta$.  At most one element in $B$ less than $i$ can be colored $\alpha$, so the elements less than $j$ in $B$ can be colored in $b(c-1)^{b-1}c^e+(c-1)^b c^e$ ways.  Anything in $B$ between $i$ and $j$ can be colored with any color.  The remaining elements less than $i$ can be colored anything except $\alpha$.  The remaining elements between $i$ and $j$ can be colored anything except $\alpha$ and $\beta$.  The elements larger than $j$ can be colored anything except $\beta$.
\end{proof}

\begin{thm} 

$\displaystyle{|\Pi_n^{eq}\wr C_c(1^\alpha1^\beta2^\alpha)|=}$

$$B(n)(c-1)^n+\sum_{k=1}^n\sum_{j=1}^k\binom{n}{k}\binom{k}{j}B(n-k)(c-1)^{n-j}+$$ $$\sum_{1\leq i<j\leq n}\sum_{a,b}\sum_{d,e}\sum_{f,g}\sum_k\sum_{p,q}\sum_{\ell}\binom{i-1}{a,b}\binom{j-i-1}{d,e}\binom{n-j}{f,g}\binom{i-a-b-1}{k}\cdot$$ $$\binom{j-i-d-e-1}{p}\binom{n-a-b-f-g-j+i-k-1}{q}\stirr{p+q}{\ell}k^{\underline{\ell}}\cdot$$ $$B(n-a-b-d-e-f-g-k-p-q-2)((c-1)^b+b(c-1)^{b-1})((c-1)^a+a(c-1)^{a-1})\cdot$$
$$(c-2)^{d+p}c^e(c-1)^{n-a-b-d-e-k-p-2},$$ where $k^{\underline{\ell}}=k\cdot(k-1)\cdots(k-\ell+1)$. \end{thm}

\begin{proof}  This proof can be broken into three cases.  Either there are no elements colored $\alpha$ there is exactly one block containing elements colored $\alpha$ or there are at least two blocks containing elements colored $\alpha$.

Suppose that no elements are colored $\alpha$.  It is then impossible to contain a copy of $1^\alpha1^\beta2^\alpha$ in the $eq$ sense.  Furthermore, there are $B(n)(c-1)^n$ such colored partitions.

Suppose that exactly one block contains elements colored $\alpha$.  Again it is impossible to contain a copy of $1^\alpha1^\beta2^\alpha$ in the $eq$ sense.  We count the number of such partitions by choosing $k$ elements to be in the block with elements colored $\alpha$ and choose $j$ of them to be colored $\alpha$.  We partition the remaining elements and since the remaining elements can be colored anything but $\alpha$ there are $(c-1)^j$ ways to color the partitions formed.  This is the second term in the first line.  

Now, suppose that at least two blocks contain elements colored $\alpha$.  Let $j$ be the largest element colored $\alpha$ and $i$ the largest element colored $\alpha$ that is not in a block with $j$.  Assume that $i$ is in block $A$ and $j$ is in block $B$.  From the elements less than $i$ we choose $a$ elements to be in block $A$ and $b$ elements to be in block $B$. Similarly we pick $d$ and $e$ elements and $f$ and $g$ elements from those between $i$ and $j$ and those greater than $j$ respectively.  We put the $d$ and $f$ elements in block $A$ and the $e$ and $g$ elements in block $B$.  Now, blocks $A$ and $B$ are formed.  

Similar to the previous proof, we will form $k$ other blocks containing elements colored $\alpha$ by choosing the minimal elements colored $\alpha$ in these blocks.  These elements must be among the elements less than $i$.  Now, choose $p$ elements from the remaining elements between $i$ and $j$ and $q$ elements from the remaining elements less than $i$ and greater than $j$ to to be in blocks with these elements.  Now we partition these $p+q$ elements into $\ell$ blocks where $0\leq \ell\leq k$, which can be done in $\stirr{p+q}{\ell}$ ways.  Now, we distribute the $k$ elements colored $\alpha$ among these $\ell$ blocks, which can be done in $k^{\underline{\ell}}$ ways.  The remaining elements are then partitions into blocks.

Now, we color the partition.  The elements in $A$ which are less than $i$ can be colored in $(c-1)^a+a(c-1)^{a-1}$ ways.  The $(c-1)^a$ is the case where no element is colored $\alpha$.  If there is an element colored $\alpha$ then no element larger than the smallest element colored $\alpha$ may be colored $\beta$.  There are $a$ choices for the location of the minimal element colored $\alpha$ and $(c-1)^{a-1}$ ways to color the remaining elements.  Similarly the elements in $B$ which are less than $i$ can be colored in $(c-1)^b+b(c-1)^{b-1}$ ways.  Each of the remaining elements less than $i$ may be colored in $c-1$ ways, since if the element is in a block with no elements colored $\alpha$ then it may not be colored $\alpha$.  If it is in a block with elements colored $\alpha$ then if it appears before the minimal element colored $\alpha$ it may not be colored $\alpha$ and if it appears after the minimal element colored $\alpha$ then it may not be colored $\beta$.  

Of the elements between $i$ and $j$, the $d$ elements in $A$ and the $p$ elements in blocks with elements colored $\alpha$ may not be colored $\alpha$ or $\beta$, so may each be colored in $c-2$ ways.  The elements in $B$ may each be colored in $c$ ways.  The elements not in any of these blocks may not be colored $\alpha$, and hence each may be colored with one of $c-1$ colors.  

Each element larger than $j$ may be colored with one of $c-1$ colors again avoiding the color $\alpha$.  Summing over appropriate values of each of the variables gives the final term of the theorem.

\end{proof}

We were unable to find a formula for $|\Pi_n\wr C_c(1^a2^b1^a)|$ and suspect that such a formula is very difficult to find.  We were able to determine a formula for $|\Pi_n\wr C_c(1^a2^b3^a)|$, which is so complicated that we have banished it to the end of the paper in an appendix.  We have also omitted the proof for the formula for $|\Pi_n\wr C_c(1^a2^b3^a)|$ because it is similar to the proofs at the end of this section.  We have found a formula for $|\Pi_n\wr C_c(1^a2^b3^d)$, but is just as complicated as the one in the appendix of this paper, so we have omitted it.  

\begin{table}[hbt]
\begin{center}
\begin{tabular}{|c|c|c|c|}
\hline
Pattern $P$&Formula& First 6 terms of $\left|\Pi_n \wr C_2(P)\right|$&OEIS number\\
\hline
$1^a2^b3^a$&Appendix&2, 8, 39, 214, 1240, 7363&new\\
\hline
$1^a2^a3^b$&Theorem 5.6&2, 8, 39, 215, 1267, 7767&new\\
\hline
$1^a2^a3^a$&Theorem 5.3&2, 8, 39, 217, 1313, 8425&new\\
\hline
$1^a2^b1^a$&&2, 8, 39, 220, 1384, 9513&new\\
\hline
$1^11^a2^b$&Theorem 5.7&2, 8, 39, 220, 1385, 9543&new\\
\hline
$1^a1^b2^a$&Theorem 5.8&2, 8, 39, 220, 1386, 9564&new\\
\hline
$1^a2^a2^a$&Theorem 5.5&2, 8, 39, 220, 1388, 9608&new\\
\hline
$1^a2^a1^a$&Theorem 5.4&2, 8, 39, 221, 1408, 9882&new\\
\hline
$1^a1^a1^a$&Corollary 5.2&2, 8, 39, 227, 1518, 11368&new\\
\hline
$1^a2^b3^d$&&2, 8, 40, 240, 1664, 12992&A055882\\
&&Note: This pattern uses 3 colors, so its&\\
&&avoidance sequence in $\Pi_n \wr C_2$ is $2^nB(n)$&\\
\hline
\end{tabular}
\caption{Partition Patterns of Length 3}
\label{Tlength3}
\end{center}
\end{table}

\section{Ideas for Future Research}

We hope this work is just the beginning of the study of pattern avoidance in colored set partitions.  As mentioned in the introduction, there are two other ways to define pattern avoidance: in the $lt$ sense, and in the $pattern$ sense.  The results in Section 3 arose from avoiding sets of patterns in the $eq$ sense that are equivalent to avoiding a single pattern in the $lt$ sense or in the $pattern$ sense.  The number of bijective results arising from these sets of patterns indicate that $lt$-avoidance and $pattern$-avoidance merit further investigation.  We are also interested in the enumeration of sets which avoid multiple patterns of length two or three that are not equivalent to $lt$ or $pattern$-type partition patterns.

We also conjecture that these enumerations can be nicely described using generating functions, and are currently investigating this idea.

Many of the same questions that have been asked for pattern-avoiding permutations may be asked for pattern-avoiding partitions:  If $\left|\Pi_n\wr C_c(P)\right|<\left|\Pi_n\wr C_c(Q)\right|$, then is $\left|\Pi_i\wr C_c(P)\right|<\left|\Pi_i\wr C_c(Q)\right|$ for all $i \geq n$?  What is the asymptotic growth of the avoidance sequences, and what of the growth of the ratios of consecutive terms?  

Finally, colored partitions add a new dimension to the avoidance problem.  It remains to explore such questions as when it is possible to have two patterns be equivalent for less than $c$ colors, but non-equivalent for $c$ or more colors.

\section{Appendix}

\label{tom}\begin{thm} The number of partitions of $\Pi_n\wr C_c$ that $eq$-avoid $(1^\alpha2^\beta3^\alpha)$ is 
$$ B(n) (c-1)^n + \sum_{i=1}^n \binom{n}{i} \sum_{j=0}^{n-i} \binom{n-i}{j} B(n-i-j) (c-1)^{n-i}+$$  $$\sum_{1\leq i<j\leq n}\sum_{a,b}\sum_{d,e}\sum_{f,g}\binom{i-1}{a,b}\binom{j-i-1}{d,e}\binom{n-j}{f,g}B(n-a-b-d-e-f-g-2)\cdot$$ $$c^{d+e}(c-1)^{n-j+i-1}(c-2)^{j-i-d-e-1}+$$ 
$$\sum_{1\leq i<j<k\leq n}\sum_{a,b}\sum_{d,e}\sum_{f,g}\sum_{\ell,m}\binom{i-1}{a,b}\binom{j-i-1}{d,e}\binom{k-j-1}{f,g}\binom{n-k}{\ell,m}\cdot$$ 

$$B(n-a-b-d-e-f-g-\ell-m-3)c^{d+e+f}(c-1)^{n-k+g+i-1}(c-2)^{k-i-d-e-f-g-2}+$$ $$\sum_{1\leq i<j<k<\ell\leq n}\sum_{a,b}\sum_{d,e}\sum_{f,g}\sum_{p,q}\sum_{r,s}\binom{i-1}{a,b}\binom{j-i-1}{d,e}\binom{k-j-1}{f,g}\binom{\ell-k-1}{p,q}\binom{n-\ell}{r,s}\cdot$$

$$B(n-a-b-d-e-f-g-p-q-r-s-2)c^{d+q}(c-1)^{n-\ell+k-j+i-2}(c-2)^{\ell-k-q+j-i-d-2}+$$ $$\sum_{1\leq i<m<\ell\leq n}\sum_{a,b}\sum_{d,e}\sum_{p,q}\sum_{r,s}\binom{i-1}{a,b}\binom{m-i-1}{d,e}\binom{\ell-m-1}{p,q}\binom{n-\ell}{r,s}\cdot$$ $$B(n-a-b-d-e-p-q-r-s-2)c^{d+q}(c-1)^{n-\ell+i-1}(c-2)^{\ell-q-i-d-2}+$$ $$\sum_{1\leq i<j<k<\ell\leq n}\sum_{a,b}\sum_{d,e}\sum_{f,g}\sum_{p,q}\sum_{r,s}2\binom{i-1}{a,b}\binom{j-i-1}{d,e}\binom{k-j-1}{f,g}\binom{\ell-k-1}{p,q}\binom{n-\ell}{r,s}\cdot$$ $$B(n-a-b-d-e-f-g-p-q-r-s-3)c^{d+q}(c-1)^{n-\ell+k-j+i-2}(c-2)^{\ell-k-q+j-i-d-2}+$$ $$\sum_{1\leq i<j<m<k<\ell\leq n}\sum_{a,b}\sum_{d,e}\sum_{f_1,g_1}\sum_{f_2,g_2}\sum_{p,q}\sum_{r,s}\binom{i-1}{a,b}\binom{j-i-1}{d,e}\binom{m-j-1}{f_1,g_1}\binom{\ell-m-1}{f_2,g_2}\cdot$$ $$\binom{\ell-k-1}{p,q}\binom{n-\ell}{r,s}B(n-a-b-d-e-f_1-g_1-f_2-g_2-p-q-r-s-4)\cdot$$ $$c^{d+q}(c-1)^{n-\ell+k-m+f_1+g_1+i-2}(c-2)^{\ell-k-q+m-i-d-f_1-g_1-3}+$$ $$\sum_{1\leq i<j<m<k<\ell\leq n}\sum_{a,b}\sum_{d,e}\sum_{f,g}\sum_{o,p}\sum_{q,r}\sum_{s,t}\binom{i-1}{a,b}\binom{j-i-1}{d,e}\binom{m-j-1}{f,g}\binom{k-m-1}{o,p}\cdot$$ $$\binom{\ell-k-1}{q,r}\binom{n-\ell}{s,t}B(n-a-b-d-e-f-g-p-q-r-s-t-3)\cdot$$
$$c^{d+q}(c-1)^{n-\ell+m-j+i+p-2}(c-2)^{j+\ell-d-i-m-p-q-3}+$$ $$\sum_{1\leq i<j<m<k<\ell\leq n}\sum_{a,b}\sum_{d,e}\sum_{f,g}\sum_{o,p}\sum_{q,r}\sum_{s,t}\binom{i-1}{a,b}\binom{j-i-1}{d,e}\binom{k-j-1}{o,p}\binom{\ell-k-1}{q,r}\cdot$$ $$\binom{n-\ell}{s,t}B(n-a-b-d-e-o-p-q-r-s-t-3)\cdot$$ $$c^{d+q}(c-1)^{n-\ell+i+p-1}(c-2)^{\ell-d-i-p-q-3}+ $$ $$\sum_{1\leq i<j<m<k<\ell\leq n}\sum_{a,b}\sum_{d,e}\sum_{f,g}\sum_{o,p}\sum_{q,r}\sum_{s,t}\binom{i-1}{a,b}\binom{j-i-1}{d,e}\binom{m-j-1}{f,g}\binom{k-m-1}{o,p}\cdot$$ $$\binom{\ell-k-1}{q,r}\binom{n-\ell}{s,t}B(n-a-b-d-e-f-g-p-q-r-s-t-4)\cdot$$
$$c^{d+q}(c-1)^{n-\ell+m-j+i+p-2}(c-2)^{j+\ell-d-i-m-p-q-3}+$$
$$\sum_{1\leq i<j<h<m<k<\ell}\sum_{a,b}\sum_{d,e}\sum_{f_1,g_1}\sum_{f_2,g_2}\sum_{o,p}\sum_{q,r}\sum_{s,t}\binom{i-1}{a,b}\binom{j-i-1}{d,e}\binom{h-j-1}{f_1,g_1}\binom{m-h-1}{f_2,g_2}\cdot$$ $$\binom{k-m-1}{o,p}\binom{\ell-k-1}{q,r}\binom{n-\ell}{s,t}B(n-a-b-d-e-f_1-g_1-f_2-g_2-o-p-q-r-s-t-5)\cdot$$ $$c^{d+q}(c-1)^{n-\ell+m-h+f_1+g_1+i+p-2}(c-2)^{h-f_1-g_1+\ell-d-i-m-p-q-4}.$$\end{thm}

\end{document}